\documentclass[11pt]{amsart}
\usepackage[top=1in, bottom=1in, left=1in, right=1in]{geometry}

\usepackage{mathtools,mathrsfs,amsthm,amsfonts,amssymb,graphicx,appendix,cancel}
\usepackage{xspace,enumerate,comment,bbm,nth}

\usepackage{tikz,pgfplots}

\usepackage{hyperref}
\usepackage[normalem]{ulem}   

\usepackage[active]{srcltx}

\numberwithin{equation}{section}
\newtheorem{theorem}{\sc Theorem}[section]

\newtheorem{definition}[theorem]{\sc Definition}
\newtheorem{lemma}[theorem]{\sc Lemma}

\theoremstyle{plain}

\newtheorem{prop}[theorem]{\sc Proposition}

\theoremstyle{remark}
\newtheorem{remark}[theorem]{\sc Remark}

 \newcommand{\ol}[1]{\overline{#1}}
 \newcommand{\ul}[1]{\underline{#1}}
 \newcommand{\wh}[1]{\widehat{#1}}

\newcommand{\F}{\mathcal F}
\newcommand{\R}{\mathbb{R}}
\newcommand{\Sol}{\mathcal S}
\newcommand{\Z}{\mathbb{Z}}
\newcommand{\N}{\mathbb{N}}

\renewcommand{\P}{\mathbb{P}}
\newcommand{\E}{\mathbb{E}}

\newcommand{\cal}[1]{\mathcal #1}
\newcommand{\1}{\mathbbm1}
\newcommand{\HV}{{\mathcal H}}

\newcommand{\eps}{\varepsilon}
\renewcommand{\epsilon}{\varepsilon}

\renewcommand{\emptyset}{\varnothing}

\renewcommand{\le}{\leqslant}
\renewcommand{\ge}{\geqslant}

\newcommand{\CC}{\D C}
\newcommand{\Lip}{\D{Lip}}

\newcommand{\ccyl}{[0,+\infty)\times\R}

\newcommand{\Ham}{\mathscr{G}}

\newcommand{\Hamzero}{\mathscr{G}_{sqc}}

\newcommand{\D}[1]{\mbox{\rm #1}}

\newcommand{\esssup}{\mathrm{ess\,sup\,}}

\newcommand{\ucv}{\rightrightarrows}

\begin{document}
	
\title[Stochastic homogenization of quasiconvex  viscous HJ
equations in 1d]{Stochastic homogenization of a class of quasiconvex\\ and possibly degenerate viscous  HJ equations in 1d}

\author[A.\ Davini]{Andrea Davini}
\address{Andrea Davini\\ Dipartimento di Matematica\\ {Sapienza} Universit\`a di
  Roma\\ P.le Aldo Moro 2, 00185 Roma\\ Italy}
\email{davini@mat.uniroma1.it}

\date{November 1, 2023}


\begin{abstract}
  We prove homogenization for possibly degenerate viscous Hamilton-Jacobi equations with a
  Hamiltonian of the form $G(p)+V(x,\omega)$, where $G$ is a quasiconvex,  
  locally Lipschitz function with superlinear growth, the
  potential $V(x,\omega)$ is bounded and Lipschitz continuous, and the diffusion coefficient 
  $a(x,\omega)$ is allowed to vanish on some regions or even on the whole $\R$. The
  class of random media we consider is defined by an explicit scaled hill
  condition on the pair $(a,V)$ which is
  fulfilled as long as the environment is not ``rigid''.
\end{abstract}

\dedicatory{Dedicated to Giuseppe Buttazzo on the occasion of his 70th birthday}

\maketitle

\section{Introduction}\label{sec:intro}

We are concerned with the asymptotic behavior,  as $\epsilon\to 0^+$, of solutions of a possibly degenerate viscous
Hamilton-Jacobi (HJ) equation of the form
\begin{equation}\label{eq:introHJ}
  \partial_t u^\epsilon=\epsilon a\left(\frac{x}{\epsilon},\omega\right) \partial^2_{xx} u ^\epsilon+G(\partial_x u^\epsilon)+  \beta V\left(\frac{x}{\epsilon},\omega\right) \qquad \hbox{in $(0,+\infty)\times\R$,}
\end{equation}
when $G:\R\to\R$ is assumed locally Lipschitz, superlinear and quasiconvex. The dependence of the equation on the random environment $(\Omega,\F,\P)$ enters through the diffusion coefficients $a$ and the potential $V$, which are assumed to be stationary with respect to shifts in $x$ and Lipschitz continuous on $\R$, for every fixed $\omega$. We furthermore assume that the functions $a$ and $V$ take values in $[0,1]$, where $\beta\geqslant 0$ is a coefficient adjusting the magnitude of the potential. We emphasize that the diffusion coefficient $a$ can vanish on some regions or even on the whole $\R$.  In this note we prove homogenization of such an equation \eqref{eq:introHJ} by additionally assuming that the pair $(a,V)$ satisfies a {\em scaled hill condition}, see Section \ref{sec:result}, to which we refer for more details and for  the precise statement of our homogenization result. 
In the inviscid case (i.e., $a\equiv 0$), this scaled hill condition reduces to requiring that $\esssup_x V(x,\omega) = 1$ almost surely, and we recover the 1-dimensional homogenization results proved in \cite{DS09}.  When $a>0$, this scaled hill condition agrees with the one adopted in \cite{Y21b} and we recover the results therein contained.\smallskip

Equations of the form \eqref{eq:introHJ} are a subclass of general
viscous stochastic HJ equations
\begin{equation}\label{vHJ}
  \partial_tu^\epsilon= \epsilon\text{tr}\left(A\left(\frac{x}{\epsilon},\omega\right) D^2 u^\epsilon\right)+H\left(D u^\epsilon,\frac{x}{\epsilon},\omega\right)\qquad \hbox{in $(0,+\infty)\times\R^d$,}
\end{equation}
where $A(\cdot,\omega)$ is a bounded, symmetric, and nonnegative
definite $d\times d$ matrix with a Lipschitz square root. The ingredients of the equation are assumed to be stationary with respect to shifts in $x$. 
This setting encompass the periodic and the quasiperiodic cases, for which homogenization has been proved under fairly general assumptions by showing the existence of (exact or approximate) correctors, i.e., sublinear functions that solve an associated stationary HJ equations \cite{LPV, Evans92, I_almostperiodic, LS_almostperiodic}. 
In the stationary ergodic setting, such solutions do not exist in general, as it was shown in \cite{LS_correctors}, see also \cite{DS09,CaSo17} for related results. This is the main reason why the extension of the homogenization theory to random media is nontrivial 
and required the development of new arguments. 

The first homogenization results in this framework were obtained for convex Hamiltonians in the case of inviscid equations in \cite{Sou99, RT} and then for their viscous counterparts in \cite{LS2005,KRV}. 
By exploiting the metric character of first order HJ equations, homogenization has been extended to the case of quasiconvex Hamiltonians, first in dimension 1 \cite{DS09} and then in any space dimension \cite{AS}. 

The topic of homogenization in stationary ergodic media for HJ equations that are nonconvex in the gradient variable
remained an open problem for about fifteen years. Recently, it has been shown via counterexamples, first in the inviscid case \cite{Zil,FS}, then in the viscous one \cite{FFZ}, that homogenization can fail for Hamiltonians of the form $H(x,p,\omega):=G(p)+V(x,\omega)$ whenever $G$ has a strict saddle point. This has shut the door to the possibility of having a general qualitative homogenization theory  in the stationary ergodic setting in dimension $d\geqslant 2$, at least without imposing further mixing conditions on the stochastic environment. 

On the positive side, a quite general homogenization result, which includes as particular instances both the inviscid and viscous cases, has been established in \cite{AT} for Hamiltonians that are positively homogeneous of degree $\alpha\geqslant 1$ and under a finite range of dependence condition on the probability space $(\Omega,\F, \P)$. 

In the inviscid case, homogenization has been proved in \cite{ATY_1d,Gao16} in dimension $d=1$ for a rather general class of coercive and nonconvex Hamiltonians, and in any space dimension for Hamiltonians of the form $H(x,p,\omega)=\big(|p|^2-1\big)^2+V(x,\omega)$, see   \cite{ATY_nonconvex}. 

Even though the addition of a diffusive term is not expected to prevent homogenization, the literature 
on viscous HJ equations with nonconvex Hamiltonians in stationary ergodic media  is more limited. 
The viscous case is in fact known to present additional 
challenges which cannot be overcome by mere
modifications of the methods used for $a\equiv 0$.

For $d=1$, apart from already mentioned work \cite{AC18},
other classes of nonconvex Hamiltonians for which the corresponding viscous HJ \eqref{vHJ} homogenizes are provided in
\cite{DK17, YZ19, KYZ20, DKY23}. In the joint paper \cite{DK17}, we have
shown homogenization of \eqref{vHJ} with $H(x,p,\omega)$ which
are ``pinned'' at one or several points on the $p$-axis and convex in
each interval in between. For example, for every $\alpha>1$ the
Hamiltonian $H(x,p,\omega)=|p|^\alpha-c(x,\omega)|p|$ 
is pinned at $p=0$ (i.e.\
$H(x,0,\omega)\equiv
\mathrm{const}$) and convex in $p$ on each of the two intervals $(-\infty, 0)$ and
$(0,+\infty)$.
Clearly, adding a non-constant potential breaks the pinning
property. In particular, homogenization of equation \eqref{vHJ}, where $d=1$, $A\equiv \mathrm{const}>0$,
\begin{equation}
  \label{open}
  H(x,p,\omega):=\frac12\,|p|^2-c(x,\omega)|p|+\beta V(x,\omega),\quad
  0<c(x,\omega)\le C,\quad \ \beta>0
\end{equation}
remained an open problem even when $c(x,\omega)\equiv c>0$.  
Homogenization for this kind of equations with $A\equiv 1/2$ and $H$ as
in \eqref{open} with $c(x,\omega)\equiv c>0$ was proved in \cite{KYZ20} 
under a novel hill and valley condition,\footnote{It is worth noticing that the valley and hill condition is fulfilled for a wide class of typical random environments without any restriction on their mixing properties, see  \cite[Example 1.3]{KYZ20} and \cite[Example 1.3]{YZ19}.  It is however not satisfied if the potential is ``rigid'', for example, in the periodic case.}
 that was introduced in \cite{YZ19} 
to study a sort of discrete version of this problem, where the Brownian motion in the  
stochastic control problem associated with equation \eqref{eq:introHJ} is replaced by controlled random walks on $\Z$. 
The approach of \cite{YZ19,KYZ20} relies on the Hopf-Cole
transformation, stochastic control representations of solutions and
the Feynman-Kac formula. It is applicable only to \eqref{eq:introHJ}
with $G(p)=\frac12|p|^2-c|p|=\min\{\frac12p^2-cp,\frac12p^2+cp\}$. In the joint paper 
\cite{DK22}, we have proposed a different proof solely based on PDE methods. 
This new approach is flexible enough to be applied to the possible degenerate case 
$a\ge 0$, and to any $G$ which is a minimum of a finite number of convex
superlinear functions $G_i$  having the same
minimum. The original hill and valley condition assumed in \cite{YZ19,KYZ20} is weakened in 
favor of a scaled hill and valley condition.
The arguments crucially rely on this condition and on the fact that all the functions $G_i$ have the same minimum. 
As in \cite{KYZ20}, the shape of the effective Hamiltonian associated with 
$G$ is derived from the ones associated to each $G_i$.

The PDE approach introduced in \cite{DK22} was subsequently refined in \cite{Y21b} in order to prove 
homogenization for equation \eqref{eq:introHJ} when the function $G$ is superlinear and quasiconvex,  $a>0$ and the pair 
$(a,V)$ satisfies a scaled hill condition equivalent to the one adopted in this paper, see Remark \ref{rem:vac}. 
The core of the proof consists in showing existence and uniqueness of sublinear correctors that satisfy suitable derivative bounds. These kind of results were obtained in \cite{DK22} by proving tailored-made  comparison principles and by exploiting a general result from \cite{CaSo17}  (and this was actually the only point where the piecewise convexity of $G$ was used). The novelty brought in by \cite{Y21b} 
relies on the nice observation that this can be directly proved via ODE arguments, which are viable since we are in one space dimension and $a>0$. 

A substantial step forward in the direction of obtaining a general homogenization result for 
viscous HJ equations in dimension 1 has been taken in \cite{DKY23}, where the  
main novelty consists in allowing a  general
superlinear $G$ without any restriction on its shape or the
number of its local extrema. Our analysis crucially relies on the assumption that the pair 
$(a,V)$ satisfies the scaled hill and valley condition and, differently from \cite{DK22}, that 
the diffusive coefficient $a$ is nondegenerate, i.e., $a>0$. 
Analogously to \cite{KYZ20, DK22}, the
function $G$ can be seen as a minimum of quasiconvex superlinear
functions $G_i$, but when these latter have distinct minima there is
a nontrivial interaction among them in the homogenization process, and
the shape of the effective Hamiltonian associated with $G$ can no
longer be guessed from those associated with each $G_i$. 
The proof consists in showing existence of suitable correctors whose derivatives 
are confined on different branches of $G$, as well as on a strong induction argument 
which uses as  base case the homogenization result established in \cite{Y21b}. 
The fact that homogenization for quasiconvex $G$ was only available for $a>0$ is one of the  
main reasons why we did not try to generalize our arguments to the possible 
degenerate case $a\geqslant 0$.\vspace{0.5ex}

The purpose of this note is to fill this gap and to provide a unified proof which encompasses both the inviscid and the viscous case. By reinterpreting the approach followed in \cite{Y21b} in the viscosity sense and by making a more substantial use of viscosity techniques, we generalize the homogenization result for equation \eqref{eq:introHJ} with quasiconvex $G$ to the possible degenerate case $a\geqslant 0$. A key tool in this regard is played by the well known stability property of the notion of viscosity solution. Our exposition also benefits of some observations and arguments developed by the author in the attempt to prove a more general homogenization result.  They have already displayed their usefulness for the joint research \cite{DKY23}, where they appeared for the first time.\vspace{0.5ex}

The paper is organized as follows. In Section \ref{sec:result} we present the setting and the standing assumptions and we state our homogenization result, see Theorem \ref{thm:genhom}. In Section \ref{sec:outline}  we describe how we can reduce the proof to the case of a quasiconvex function $G$ satisfying more stringent assumptions, see Theorem \ref{teo reduction}, and we outline our strategy to prove Theorem \ref{thm:genhom}. In Section \ref{sec:PDE results} we prove some deterministic PDE results for the stationary HJ equations we shall consider in the sequel. In Section \ref{sec:correctors} we prove existence and uniqueness of correctors possessing stationary derivatives that satisfy suitable bounds. We also derive some additional information that will be needed in Section \ref{sec:homogenization}, where we prove Theorem \ref{thm:genhom}.\\

\indent{\textsc{Acknowledgements. $-$}} The author benefited from on-line discussions with Elena Kosygina and Atilla Yilmaz in the period April-July 2020. The author wishes to thank Luca Rossi for the help provided with the proof of Proposition \ref{prop pointwise solution}. The author also thanks the anonymous referee for his comments, who contributed to improve the presentation.


%
%

\section{Let us get started}\label{sec:preliminaries}
\subsection{Assumptions and statement of the main result}\label{sec:result}

We will denote by $\CC(\R)$ the family of continuous functions  on $\R$. We will regard $\CC(\R)$ as a Polish space endowed with a metric inducing the topology of uniform convergence on compact subsets of $\R$. 

The triple $(\Omega,\F, \P)$ denotes a probability space, where $\Omega$ is a Polish space, ${\cal F}$ is the $\sigma$-algebra of Borel subsets of $\Omega$, and $\P$ is a complete probability measure on $(\Omega,{\cal F})$.\footnote{The assumptions that $\Omega$ is a Polish space and $\P$ is a complete probability measure are only used in the proof of Theorem \ref{teo1 correctors} to show that the functions $u^\lambda_\pm$ are jointly measurable in $\R\times\Omega$.} We will denote by ${\cal B}$ the Borel $\sigma$-algebra on $\R$ and equip the product space $\R\times \Omega$ with the product $\sigma$-algebra ${\mathcal B}\otimes {\cal F}$.

We will assume that $\P$ is invariant under the action of a one-parameter group $(\tau_x)_{x\in\R}$ of transformations $\tau_x:\Omega\to\Omega$. More precisely, we assume that the mapping
$(x,\omega)\mapsto \tau_x\omega$ from $\R\times \Omega$ to $\Omega$ is measurable, $\tau_0=id$, $\tau_{x+y}=\tau_x\circ\tau_y$ for every $x,y\in\R$, and $\P\big(\tau_x (E)\big)=\P(E)$ for every $E\in{\cal F}$ and $x\in\R$. We will assume in addition that the action of $(\tau_x)_{x\in\R}$ is {\em ergodic}, i.e., any measurable function $\varphi:\Omega\to\R$ satisfying $\P(\varphi(\tau_x\omega) = \varphi(\omega)) = 1$ for every fixed $x\in{\R}$ is almost surely equal to a constant.
 If $\varphi\in L^1(\Omega)$, we write $\E(\varphi)$ for the mean of $\varphi$ on $\Omega$, i.e. the quantity
$\int_\Omega \varphi(\omega)\, d \P(\omega)$.

A measurable function $f:\R\times \Omega\to \R$ is said to be {\em stationary} with respect to $(\tau_x)_{x\in\R}$ if  $f(x,\omega)=f(0,\tau_x\omega)$ for all $(x,\omega)\in\R\times\Omega$. Moreover, whenever the action of $(\tau_x)_{x\in\R}$ is ergodic, we refer to $f$ as a stationary ergodic function.\smallskip

In this paper, we will consider an equation of the form 
\begin{equation}\label{eq:generalHJ}
\partial_t u=a(x,\omega) \partial^2_{xx} u +G(\partial_x u)+  \beta V(x,\omega),\quad (t,x)\in (0,+\infty)\times\R,
\end{equation}
where $\beta\geqslant 0$, and $a:\R\times\Omega\to [0,1]$, $V:\R\times\Omega\to [0,1]$ are stationary ergodic processes satisfying the following assumptions, 
%
%
for some constant $\kappa > 0$:
\begin{itemize}
\item[(A)]  $\sqrt{a(\,\cdot\,,\omega)}:\R\to [0,1]$\ is $\kappa$--Lipschitz continuous for all $\omega\in\Omega$;\footnote{
Note that (A) implies that $a(\cdot,\omega)$ is $2\kappa$--Lipschitz in $\R$ for all $\omega\in\Omega$.  Indeed, 
for all $x,y\in\R$ we have 
\[
|a(x,\omega) - a(y,\omega)| = |\sqrt{a(x,\omega)} + \sqrt{a(y,\omega)}||\sqrt{a(x,\omega)} - \sqrt{a(y,\omega)}| \leqslant 2\kappa|x-y|.
\]
}
\medskip
\item[(V)] $V(\,\cdot\,,\omega):\R\to [0,1]$\ is $\kappa$--Lipschitz
  continuous for all $\omega\in\Omega$.\\
\end{itemize}

As for the nonlinearity $G$, we assume that it belongs to the class $\Ham$ defined as follows. 
\begin{definition}\label{def:Ham}
	A function $G:\R\to\R$ is said to be in the class $\Ham(\alpha_0,\alpha_1,\gamma)$ if it satisfies the following conditions, for some constants $\alpha_0,\alpha_1>0$ and $\gamma>1$:
	\begin{itemize}
		\item[(G1)] $\alpha_0|p|^\gamma-1/\alpha_0\leqslant G(p)\leqslant\alpha_1(|p|^\gamma+1)$\ for all ${p\in\R}$;\medskip
		\item[(G2)] $|G(p)-G(q)|\leqslant\alpha_1\left(|p|+|q|+1\right)^{\gamma-1}|p-q|$\ for all $p,q\in\R$.\smallskip
\end{itemize}
We will denote by $\Ham$ the union of the families $\Ham(\alpha_0,\alpha_1,\gamma)$, where $\alpha_0,\alpha_1$ vary in $(0,+\infty)$ and $\gamma$ in $(1,+\infty)$. 
\end{definition}

Solutions, subsolutions and supersolutions of \eqref{eq:generalHJ}
will be always understood in the viscosity sense, see
\cite{CaCa95, users,barles_book,bardi}. Assumptions (A), (V), (G1)-(G2)
guarantee well-posedness in $\D{UC}(\ccyl)$ of the Cauchy problem for
the parabolic equation \eqref{eq:generalHJ}, as well as Lipschitz
estimates for the solutions under appropriate assumptions on the
initial condition, see \cite[Appendix A]{DK22} for more
details.

      The purpose of this paper is to establish a homogenization
      result for equation \eqref{eq:introHJ} when the nonlinearity $G$ is quasiconvex, namely 
\begin{itemize}
\item[(qC)] \quad $\{p\in\R\,:\,G(p)\leqslant \lambda\}$ is a (possibly empty) convex set for every $\lambda\in\R$. 
\end{itemize}
We succeeded in doing this under the additional hypothesis that the pair $(a,V)$ satisfies the following condition, 
termed {\em scaled hill condition}:
\begin{itemize}
\item[(S)]  for every $h\in (0,1)$ and $y>0$, there exists a set $\Omega(h,y)$ of probability 1 such that, 
for every $\omega\in \Omega(h,y)$, 
there exist $\ell_1<\ell_2$ and $\delta>0$ (depending on $\omega$) such that 
\begin{itemize}
\item[(a)] \quad $\displaystyle \int_{\ell_1}^{\ell_2} \frac{dx}{a(x,\omega)\vee\delta} = y$,\medskip
\end{itemize}
\begin{itemize}
\item[(h)]\ {\em (hill)} \quad $V(\,\cdot\,,\omega)\geqslant h\quad\hbox{on $[\ell_1,\ell_2]$}$.\medskip
\end{itemize}
\end{itemize}

\begin{remark}\label{rem:vac}
When $a>0$, the above scaled hill condition is equivalent to the one adopted in \cite{Y21b}. When 
$a\equiv 0$, it reduces to requiring that ${\esssup_{x\in\R} V(x,\omega) = 1}$ almost surely, since item (a) is trivially satisfied by any interval $(\ell_1,\ell_2)$ provided $\delta>0$ is chosen small enough.
%
%
\end{remark}

%
%

Our main result reads as follows.

\begin{theorem}\label{thm:genhom}
	Suppose $a$ and $V$ satisfy (A), (V) and (S), $\beta\geqslant 0$, and $G\in\Ham$ satisfies (qC). Then, 
the viscous HJ equation \eqref{eq:introHJ} homogenizes, i.e., there exists a continuous and coercive function 
$\HV(G):\R\to\R$, called {\em effective Hamiltonian}, and a set $\hat\Omega$ of probability 1 such that, for every uniformly continuous function $g$ on $\R$ and every $\omega\in\hat\Omega$, the solutions $u^\epsilon(\cdot,\cdot,\omega)$ of \eqref{eq:introHJ} satisfying $u^\epsilon(0,\,\cdot\,,\omega) = g$ converge,
    locally uniformly on $[0, +\infty)\times \R$ as $\epsilon\to 0^+$, to the unique solution $\ol{u}$ of 
    \begin{eqnarray*}
   \begin{cases}
    \partial_t \ol{u} = \HV(G)(D\ol{u}) & \hbox{in $(0,+\infty)\times\R$}\\
    \ol{u}(0,\,\cdot\,) = g & \hbox{in $\R$}.
    \end{cases}
    \end{eqnarray*}
Furthermore, $\HV(G)$ belongs to $\Ham$, it is quasiconvex, $\HV(G)(\R)=[\min (G+\beta),+\infty)$ and 
$$
\HV(G)^{-1}\big(\min(G+\beta)\big)=[\theta_-,\theta_+]
\qquad 
\hbox{for some\ \ $\theta_-\leqslant \theta_+$.}
$$ 
\end{theorem}

We remark that the effective Hamiltonian $\HV(G)$ also depends on the pair $(a, V)$ and the constant $\beta$. 
We do not keep track of this in our notation since they all remain fixed throughout the paper.

\subsection{Strategy outline  and a reduction}\label{sec:outline}

In this section, we will describe our strategy to prove the homogenization result stated in Theorem \ref{thm:genhom}. 
Let us denote by $u_\theta(\,\cdot\,,\,\cdot\,,\omega)$ the unique Lipschitz solution to \eqref{eq:generalHJ} with initial condition $u_\theta(0,x,\omega)=\theta x$ on $\R$, and let us introduce the following deterministic quantities, defined almost surely on $\Omega$, see \cite[Proposition 3.1]{DK22}:
\begin{eqnarray}\label{eq:infsup}
	\HV^L(G) (\theta):=\liminf_{t\to +\infty}\ \frac{u_\theta(t,0,\omega)}{t}\quad\text{and}\quad 
	\HV^U(G) (\theta):=\limsup_{t\to +\infty}\ \frac{u_\theta(t,0,\omega)}{t}.
\end{eqnarray}
In view of \cite[Lemma 4.1]{DK17}, in order to prove homogenization, it suffices to show that $\HV^L(G)(\theta) = \HV^U(G)(\theta)$ for every $\theta\in\R$. If this occurs, their common value is denoted by $\HV(G)(\theta)$. The function $\HV(G):\R\to\R$ is called the effective Hamiltonian associated with $G$. It has already appeared in the statement of Theorem \ref{thm:genhom}.

Note that the above observation readily implies homogenization when $\beta=0$. Indeed, in this instance, $u_\theta(t,x,\omega)=\theta x + tG(\theta)$ for all $(t,x)\in\ccyl$ and $\theta \in\R$, so $\HV(G) = G$. Hence, to prove our homogenization result, it suffices to consider $\beta>0$. \smallskip

The first step  consists in reducing to the case when $G$ belongs 
to the following subclass of $\Ham$.

\begin{definition}\label{def:Hamtwo}
	A function $G\in\Ham$ is said to be in the class $\Hamzero$ (where the subscript {\em sqc} stands for {\em strictly quasi--convex}) if it satisfies the following additional conditions, for some $\eta>0$:
	\begin{itemize}
		\item[(G3)] $G(0)= 0$;\smallskip
		\item[(G4)]  for every $p_2>p_1\geqslant 0$ 
		\[
		G(p_2)-G(p_1)\geqslant \eta |p_2-p_1|
		\quad
		\hbox{and}
		\quad
		G(-p_2)-G(-p_1)\geqslant \eta |p_2-p_1|
		\]	
		\end{itemize}
\end{definition}
\smallskip

This is motivated by the following fact.\smallskip

\begin{theorem}\label{teo reduction}
If Theorem \ref{thm:genhom} holds for every $G\in\Hamzero$, then it holds for every $G\in\Ham$.
\end{theorem}

\begin{proof}
Pick $G\in\Ham$. By (G1)--(G2),  there exists a $p_{\rm{min}}\in\R$ such that 
$G(p_{\rm{min}})=\min G$. Let us set $\widetilde G(\,\cdot\,):=G(\,\cdot\, + p_{\rm{min}}) - \min G$. An easy computation shows that Theorem \ref{thm:genhom}  holds for $G$ if and only if it holds for $\widetilde G$, and, in this instance, 
we have $\HV(\widetilde G)(\theta)=\HV(G)(\theta+ p_{\rm{min}})-\min G$ for all $\theta\in\R$. Indeed, if $\tilde u_\theta$ denotes the the unique Lipschitz solution to \eqref{eq:generalHJ} with $\tilde G$ in place of $G$ and initial condition $\tilde u_\theta(0,x,\omega)=\theta x$ on $\R$, then the function $u:=\tilde u_\theta +p_{\rm{min}}x+t\min G$ solves equation \eqref{eq:generalHJ} with initial condition $u(0,x,\omega)=(\theta+p_{\rm{min}})x$. Hence, without any loss of generality, it suffices to prove the assertion for a $G\in\Ham$ additionally satisfying $G(p)\geqslant G(0)=0$ for all $p\in\R$. For every $n\in\N$, let us set $G_n(p):=G(p)+|p|/n^2$ for all $p\in\R$. Then the functions $G$ and $(G_n)_n$ belong to $\Ham(\alpha_0,\alpha_1,\gamma)$ for some constants $\alpha_0,\alpha_1>0$ and $\gamma>1$. Furthermore, 
$G_n\in\Hamzero$, hence equation \eqref{eq:introHJ} homogenizes with $G:=G_n$ for each $n\in\N$. Since  
$\|G-G_n\|_{L^\infty([-n,n])}<1/n$ for all $n\in\N$, we infer that \eqref{eq:introHJ} homogenizes for $G$ as well 
in view of \cite[Theorem B.4]{DKY23}. Furthermore, the effective Hamiltonians $\HV(G_n)$ all belong to   
$\Ham(\overline\alpha_0,\overline\alpha_1,\overline\gamma)$ for possible different constants $\overline\alpha_0,\overline\alpha_1>0$ and $\overline\gamma>1$, and converge to $\HV(G)$, locally uniformly in $\R$, as $n\to +\infty$, see Proposition B.3 and Theorem B.4 in \cite{DKY23}. This implies that $\HV(G)$ belongs to 
$\Ham(\overline\alpha_0,\overline\alpha_1,\overline\gamma)$ as well and it inherits  from the functions $\HV(G_n)$ the properties listed in the remainder of the statement of Theorem \ref{thm:genhom}. 
\end{proof}

Let us therefore assume that $G\in\Hamzero$. Let us denote by $G_-:(-\infty,0]\to [0,+\infty)$ and 
$G_+:[0,+\infty)\to [0,+\infty)$ the functions defined as follows:
\[
G_-(p):=G(p)\quad\hbox{for all $p\leqslant 0$,}
\qquad
G_+(p):=G(p)\quad\hbox{for all $p\geqslant 0$.}
\]
Condition (G4) means that $G_-$ is strictly decreasing on $(-\infty,0]$ and $G_+$ is strictly increasing on $[0,+\infty)$ with a fixed rate of monotonicity.\medskip

In order to prove that $\HV^L(\theta)=\HV^U(\theta)$ for all $\theta\in\R$, 
we will adopt the approach that was taken in \cite{KYZ20, DK22} and subsequently developed in \cite{Y21b}.
It consists in showing the existence of a viscosity (Lipschitz) solution $u(x,\omega)$ with stationary derivative
for the following stationary equation associated with \eqref{eq:generalHJ}, namely 
\begin{equation}\label{eq0:cellODE}
	a(x,\omega)u''+G(u')+\beta V(x,\omega)=\lambda\qquad\hbox{in $\R$}
\end{equation}
 for $\P$-a.e. $\omega\in\Omega$. Such solutions are called {\em correctors} for the role they play in homogenization. In fact, the following fact holds.

\begin{prop}\label{prop consequence existence corrector}
Let $\lambda\in\R$ such that equation \eqref{eq0:cellODE} admits a Lipschitz solution $u(x,\omega)$ with stationary derivative. 
Let us set $\theta:=\E[u'(0,\omega)]$.  Then \ $\HV^L(\theta)=\HV^U(\theta)=\lambda$. 
\end{prop}

\begin{proof}
Let us set $F_\theta(x,\omega):=u(x,\omega)-\theta x$ for all $(x,\omega)\in\R\times\Omega$. Then $F_\theta(\cdot,\omega)$ 
is sublinear for $\P$-a.e. $\omega\in\Omega$, see for instance \cite[Theorem 3.9]{DS09}. Furthermore, it is a viscosity solution of \eqref{eq0:cellODE} with $G(\theta+\cdot)$ in place of $G$. The assertion follows in view of \cite[Lemma 5.6]{DK22}.
\end{proof}

We proceed by recalling a result that is proved in  \cite[Proposition 4.1]{DK22}. It crucially relies on the scaled hill condition (S). 

\begin{prop}\label{prop lower bound}
Let $a$ and $V$ satisfy conditions (A), (V) and (S) and $G\in\Ham$. Then 
\begin{equation*}
\HV^L(G)(\theta)\geqslant \min_\R (G+\beta)
\qquad
\hbox{for all $\theta\in\R$.}
\end{equation*}
\end{prop}

Our strategy to prove Theorem \ref{thm:genhom} for $G\in\Hamzero$ can be outlined as follows. First we remark that, in view of Propositions \ref{prop consequence existence corrector} and \ref{prop lower bound}, we need to look for correctors for $\lambda\geqslant \beta$ only. 
For every fixed $\omega\in\Omega$ and $\lambda\geqslant \beta$, we will show the existence of a unique pair of Lipschitz solutions $u_-^\lambda(\cdot,\omega)$ and $u_+^\lambda(\cdot,\omega)$ to \eqref{eq0:cellODE} satisfying $u_-^\lambda(0,\omega)=u_+^\lambda(0,\omega)=0$ and the following derivative bounds, see Theorem \ref{teo1 correctors}:
\begin{equation*}
(u^\lambda_-)'(x,\omega)\in [G_-^{-1}(\lambda), G_-^{-1}(\lambda-\beta)],
\quad
(u^\lambda_+)'(x,\omega)\in [G_+^{-1}(\lambda-\beta), G_+^{-1}(\lambda)]
\qquad
\hbox{for a.e. $x\in\R$}.
\end{equation*}
This result is deterministic, with $\omega$ treated as a fixed parameter, and does not depend on the scaled hill condition (S). The uniqueness part, which is proved in Proposition \ref{prop pre-uniqueness}, relies on a standard application of Gronwall's Lemma and depends crucially on the fact that $G_-$ and $G_+$ are strictly monotone on their domains of definition. Such a uniqueness property readily implies that the functions 
$u^\lambda_\pm$ have stationary derivatives. Next, we need show that the dependence of the functions $u^\lambda_\pm$ on $\omega$ is measurable. This is a technical but important point. For its proof, we exploit the fact that $\Omega$ is a Polish space and $\P$ is a complete probability measure, see Theorem \ref{teo1 correctors}, and this is the only point in the paper where this condition is used. All this allows us to define two functions $\theta_\pm: [\beta,+\infty)\to\R$ defined as
$\theta_-(\lambda):=\E[\partial_x u_-^\lambda(0,\omega)]$ and $\theta_+(\lambda):=\E[\partial_x u_+^\lambda(0,\omega)]$. These functions are shown to be strictly monotone and Lipschitz, see Proposition \ref{prop properties theta}. Hence, for every 
$\theta\in (-\infty,\theta_-(\beta)]\cup [\theta_+(\beta),+\infty)$ we either have $\theta=\theta_-(\lambda)$ or $\theta=\theta_+(\lambda)$ for some $\lambda\geqslant \beta$. Since, by Proposition \ref{prop consequence existence corrector}, we have$\HV^L(G)(\theta_{\pm}(\lambda))=\HV^U(G)(\theta_{\pm}(\lambda))=\lambda$, we derive that we can define $\HV(G)$ on $(-\infty,\theta_-(\beta)]\cup [\theta_+(\beta),+\infty)$ by setting 
\begin{equation}\label{def HV(G)}
\HV(G)(\theta):=\HV^L(G)(\theta)=\HV^U(G)(\theta)
\qquad
\hbox{for all $\theta \in  (-\infty,\theta_-(\beta)]\cup [\theta_+(\beta),+\infty)$.}
\end{equation}
In order to conclude, we have to extend the definition of $\HV(G)$ to the whole $\R$ by showing that the second equality in \eqref{def HV(G)} holds for $\theta\in \big(\theta_-(\beta), \theta_+(\beta)\big)$ too. This is done in Proposition \ref{prop flat part}, where we show that $\HV^L(G)(\theta)=\HV^U(G)(\theta)=\beta$\ for every $\theta\in \big(\theta_-(\beta),\theta_+(\beta)\big)$, and this holds true because the scaled hill condition (S) is in force.\medskip

\section{Deterministic viscous HJ equations}\label{sec:PDE results}

In this section we collect some deterministic results for a  viscous HJ equation of the form
\begin{equation}\label{eq PDE}
	a(x)u''(x)+G(u')+\beta V(x)=\lambda\qquad \hbox{ in $\R$,}
\end{equation}
where $\lambda\in\R$, the nonlinearity $G$ belongs to $\Ham$, and  $a,\, V$ are deterministic functions satisfying conditions (A) and (V), respectively. 
The following holds. 
\begin{prop}\label{prop regularity solutions}
Let $u\in\CC(\R)$ be a viscosity solution of \eqref{eq PDE}. Let us assume that $G\in\Ham(\alpha_0,\alpha_1,\gamma)$ for some constants $\alpha_0,\alpha_1>0$ and $\gamma>1$. Then $u$ is $K$--Lipschitz in $\R$ for some constant $K=K(\lambda,\beta, \gamma,\alpha_0,\alpha_1,\kappa)$ only depending on $\lambda, \beta,\gamma,\alpha_0,\alpha_1,\kappa$.\footnote{We recall that $\kappa$ is the Lipschitz constant appearing in conditions (A) and (V).} Furthermore, $u$ is of class $C^2$ (and hence a pointwise solution of \eqref{eq PDE}) in every open interval $I$ where $a(\cdot)$ is strictly positive. 
\end{prop}

\begin{proof}
The Lipschitz character of $u$ is direct consequence of \cite[Theorem 3.1]{AT}, to which we refer for an explicit expression of the dependence of $K$ on the parameters  $\lambda, \beta,\gamma,\alpha_0,\alpha_1,\kappa$.  

Let us now assume that $a(\cdot)$ is strictly positive on some open interval $I$. Without loss of generality, we can assume that $I$ is bounded and $\inf_I a>0$. From the Lipschitz character of $u$ we infer that $-C\leq u''\leq C$ in $I$ in the viscosity sense 
for some constant $C>0$, or, equivalently, in the distributional sense, in view of  \cite{Is95}. Hence, 
$u'' \in L^\infty(I)$. The elliptic regularity theory, see \cite[Corollary 9.18]{GilTru01}, ensures that $u\in W^{2,p}(I)$ for any $p>1$
and, hence, $u\in \D{C}^{1,\sigma}(I)$ for any $0<\sigma<1$. Since $u$ is a viscosity solution to \eqref{eq PDE} in $I$, 
by Schauder theory \cite[Theorem 5.20]{HL97}, we conclude that $u\in \D{C}^{2,\sigma}(I)$ for any $0<\sigma<1$.\,
\end{proof}

Proposition \ref{prop regularity solutions} together with the next result implies in particular that any viscosity solution of \eqref{eq PDE} solves the equation at almost every point.

\begin{prop}\label{prop pointwise solution}
Let $u\in\CC(\R)$ be a viscosity supersolution (respectively, subsolution) of \eqref{eq PDE}. If $u$ is differentiable at a point $x_0\in\R$ where  $a(x_0)=0$, then 
\[
G(u'(x_0))+\beta V(x_0)\leqslant \lambda \qquad \hbox{(resp., $G(u'(x_0))+\beta V(x_0)\geqslant \lambda$).}
\]
\end{prop}

\begin{proof}
To ease notation, denote by $H(x,p):=G(p)+\beta V(x)-\lambda$. Without any loss in generality, we can assume $x_0=0$ and $u(0)=u'(0)=0$. Then 
\begin{equation}\label{eq small o}
|u(x)|\leqslant |x|\omega(|x|)\qquad\hbox{in a neighborhood of $x=0$,}
\end{equation} 
where $\omega$ is a continuity modulus. Let us assume that $u$ is a supersolution. Fix $\eps>0$ and $r>0$ and  set $\varphi_r(x):=-\eps {x^2}/{r}$. By \eqref{eq small o}, there exists $r_0>0$ such that $\omega(r_0)<\eps$. For every $r\leqslant r_0$ we have 
\[
\max_{\partial B_r} (\varphi_r-u) 
\leqslant
r(-\eps +\omega(r))
<0=(\varphi_r-u)(0),
\]
hence there exists $x_r\in B_r$ such that $\varphi_r-u$ attains a local maximum at $x_r$. Being $u$ a supersolution of \eqref{eq PDE}, we infer
\[
-a(x_r)\dfrac{2\eps}{r}+H\left(x_r,-2\eps\frac{x_r}{r}\right)\leqslant 0.
\]
Now $|x_r/r|\leqslant 1$ and 
$\displaystyle \left|\frac{a(x_r)}{r}\right| =  \left|\frac{a(x_r)-a(0)}{r}\right|\leqslant 2\kappa  \left|\frac{x_r}{r}\right|\leqslant 2\kappa$. By sending $r\to 0^+$ and by extracting convergent subsequences, we can find $|p_\eps|\leqslant 1$ and $|\alpha_\eps|\leqslant 2\kappa$ such that
\[
2\eps \alpha_\eps +H(0,2\eps p_\eps)\leqslant 0.
\]
By sending $\eps\to 0^+$ we finally obtain $H(0,0)\leqslant 0$, as it was to be shown. The case when $u$ is a subsolution can be handled analogously.
\end{proof}

Let us now assume that $G\in\Hamzero$. For any fixed $\lambda\geqslant \beta$, we denote by
\begin{eqnarray*}
\Sol^\lambda_-&:=& \{u\in\Lip(\R)\,:\, \text{$u$ solves \eqref{eq PDE} and 
$u'(x)\in [G_-^{-1}(\lambda), G_-^{-1}(\lambda-\beta)]$ for a.e. $x\in\R$}\};\\
\Sol^\lambda_+ &:=& \{u\in\Lip(\R)\,:\, \text{$u$ solves \eqref{eq PDE} and 
$u'(x)\in [G_+^{-1}(\lambda-\beta), G_+^{-1}(\lambda)]$ for a.e. $x\in\R$}\}.\medskip
\end{eqnarray*}

\begin{prop}\label{prop pre-uniqueness}
Let $I$ be an open interval of $\R$ such that $a(\cdot)>0$ in $I$. Assume that 
$a=0$ on $\partial I$ when this set is nonempty. Let $\lambda\geqslant \beta$ and $u_1,u_2\in\Sol_+^\lambda$ (respectively,  $u_1,u_2\in\Sol_-^\lambda$). Then $u'_1(x)=u'_2(x)$ for a.e. $x\in I$. 
\end{prop}

For the proof, we will need the following auxiliary lemma.

\begin{lemma}\label{lemma infinite integral}
Let $I$ be an open interval in $\R$ such that $a>0$ in $I$.  Assume that 
$a=0$ on $\partial I$ when this set is nonempty. For every $x_0\in I$, 
we have 
\[
\int_{I\cap (-\infty,x_0)} \frac{1}{a(x)} \, dx=+\infty
\qquad
\hbox{and}
\qquad
\int_{I\cap (x_0,+\infty)} \frac{1}{a(x)} \, dx=+\infty
\]
\end{lemma}

\begin{proof}
When $I=\R$, the assertion follows from the inequality $a(\cdot)\leqslant 1$ on $\R$. 
Next, let us consider the case when $I$ is bounded, i.e. $I=(\ell_1,\ell_2)$  for some $\ell_1<\ell_2$,  and $a(\ell_1)=a(\ell_2)=0$. 
Let us fix $x_0\in (\ell_1,\ell_2)$. For every $x\in I\setminus\{x_0\}$ we have
\[
\frac{1}{a(x)}
=
\frac{1}{a(x)-a(\ell_2)}
\geqslant 
\frac{1}{2\kappa |x-\ell_2|},
\qquad
\frac{1}{a(x)}
=
\frac{1}{a(x)-a(\ell_1)}
\geqslant 
\frac{1}{2\kappa |x-\ell_1|},
\]
By integrating on $(x_0,\ell_2)$ and on $(\ell_1,x_0)$, respectively, we get the assertion.
The remaining cases can be treated analogously. 
\end{proof}

\begin{proof}[Proof of Proposition \ref{prop pre-uniqueness}]
We prove the assertion for $u_1,u_2\in\Sol_+^\lambda$, being the other case analogous. The functions $u_1,u_2$ are of class $C^2$ in $I$ by Proposition \ref{prop regularity solutions}. We infer that the functions $f_i:=u'_i$ for $i=1,2$ satisfy the inequality 
$0\leqslant G_+^{-1}(\lambda-\beta)\leqslant f_i(x)\leqslant G_+^{-1}(\lambda)$\ for all $x\in I$ and are classical solutions of the ODE 
\begin{equation}\label{eq local ODE}
a(x)f'+G(f)+\beta V(x)=\lambda\qquad\hbox{in $I$.}
\end{equation}
Since $G$ is locally Lipschitz, we infer that either $f_1\equiv f_2$, or the graphs of the functions $f_1,f_2$ cannot cross. Let us assume for definiteness that $f_2>f_1$ in $I$. From \eqref{eq local ODE} we get
\[
0=a(x)(f_2'-f_1')+G(f_2)-G(f_1)\geqslant a(x)(f_2-f_1)'+\eta (f_2-f_1).
\]
We proceed by arguing as in the proof of Lemma 3.2 in \cite{Y21b}. 
Let us set $w:=f_2-f_1$ and let $x>x_0$ be a pair of points arbitrarily chosen
in $I$.  By Gronwall's Lemma we infer 
\begin{equation}\label{eq Gronwall}
0
\leqslant 
w(x)
\leqslant 
e^{-\int_{x_0}^x\frac{\eta}{a(z)} dz } w(x_0)
\leqslant 
e^{-\int_{x_0}^x\frac{\eta}{a(z)} dz } \left(G_+^{-1}(\lambda)-G_+^{-1}(\lambda-\beta)\right).
\end{equation}
Let us set $\ell_1:=\inf I$, where we agree that $\ell_1=-\infty$ when $I$ is not bounded from below. By sending $x_0\to \ell_1^+$ in \eqref{eq Gronwall}, we conclude that $w(x)=0$, i.e., $f_1=f_2$ in $I$ since $x$ was arbitrarily chosen in $I$.
\end{proof}

\section{About correctors}\label{sec:correctors}

Let us consider the following viscous HJ equation
\begin{equation}\label{eq:cellODE}
	a(x,\omega)u''(x)+G(u')+\beta V(x,\omega)=\lambda\qquad \hbox{ in $\R$.}
\end{equation}
Throughout this section, we will assume $a,V:\R\times\Omega\to\R$ to be stationary functions satisfying (A) and (V), respectively, and $G$ to belong to $\Hamzero$.

For $\lambda\geqslant \beta$ and $\omega\in\Omega$ fixed, we denote by
\begin{eqnarray*}
\Sol^\lambda_-(\omega) &:=& \{u\in\Lip(\R)\,:\, \text{$u$ solves \eqref{eq:cellODE} and 
$u'(x)\in [G_-^{-1}(\lambda), G_-^{-1}(\lambda-\beta)]$ for a.e. $x\in\R$}\},\\
\Sol^\lambda_+(\omega) &:=& \{u\in\Lip(\R)\,:\, \text{$u$ solves \eqref{eq:cellODE} and 
$u'(x)\in [G_+^{-1}(\lambda-\beta), G_+^{-1}(\lambda)]$ for a.e. $x\in\R$}\}.\bigskip
\end{eqnarray*}

\noindent We proceed to show the following fact. 

\begin{theorem}\label{teo1 correctors}
Let $\lambda\geqslant\beta$  be fixed. The following holds:
\begin{itemize}
\item[(i)] for every $\omega\in\Omega$, there exists a unique $u_+^\lambda\in\Sol^\lambda_+(\omega)$ with $u_+^\lambda(0,\omega)=0$;\smallskip
\item[(ii)] for every $\omega\in\Omega$, there exists a unique $u_-^\lambda\in\Sol^\lambda_-(\omega)$ with $u_-^\lambda(0,\omega)=0$.\medskip
\end{itemize}
Furthermore, the functions $u^\lambda_+,u^\lambda_-:\R\times\Omega\to\R$ are jointly measurable and have stationary derivatives. 
\end{theorem}

\begin{proof}
We only prove the assertion for $u^\lambda_+$, being the other case analogous. Let us fix $\omega\in\Omega$ and let $u_1,u_2\in\Sol_+^\lambda(\omega)$. In view of Propositions \ref{prop regularity solutions}. and \ref{prop pointwise solution}, and the fact that $u'_1,u_2'\geqslant 0$ a.e. in $\R$, we have 
\[
u_1'(x)=G_+^{-1}\left(\lambda - \beta V(x,\omega)\right)=u_2'(x)
\qquad
\hbox{for a.e. $x\in\{y\in\R\,:\,a(y,\omega)=0\}$}.
\]
Let $I$ be a connected component of the set $\{y\in\R\,:\,a(y,\omega)>0\}$. Then $I$ is an open interval of $\R$  and 
$a=0$ on $\partial I$ when this set is nonempty. By Proposition \ref{prop pre-uniqueness} we infer that $u'_1(x)=u'_2(x)$ for every $x\in I$. We conclude that $u_1'(x)=u_2'(x)$ for a.e. $x\in\R$. That proves the uniqueness part. 

Let us show that $\Sol_+^\lambda(\omega)\not=\emptyset$. 
Set $a_n(\cdot,\omega):=a(\cdot,\omega)\vee 1/n$ and denote by ${u^\lambda_n}_+$ the unique solution of 
equation \eqref{eq:cellODE} with $a_n$ in place of $a$ which satisfies $u_n(0)=0$ and 
\begin{equation}\label{eq Lipschitz bound}
u_n'(x)\in [G_+^{-1}(\lambda-\beta), G_+^{-1}(\lambda)]\qquad\hbox{for all $x\in\R$}.
\footnote{Note that $u_n\in C^2(\R)$ in view of Proposition \ref{prop regularity solutions}.}
\end{equation}
The existence of such a solution $u_n$ is guaranteed by 
\cite[Lemma 3.2]{Y21b}.\footnote{Set $u_n(x):=\int_0^x f^\lambda_2(z) dz$, with $f^\lambda_2$ as in \cite[Lemma 3.2]{Y21b}.}  
The functions $(u_n)_n$ are equi-Lipschitz and locally equi-bounded, so they converge in $\CC(\R)$, up to extracting a subsequence, to a function $u$ satisfying $u(0)=0$ and \eqref{eq Lipschitz bound} for almost every $x\in\R$. Since $a_n(\cdot,\omega)\ucv a(\cdot,\omega)$ in $\R$, by the stability of the notion of viscosity solution we conclude that $u$ solves \eqref{eq:cellODE} in the viscosity sense.\footnote{We have proved above that such a solution $u$ is unique, hence the whole sequence $(u_n)_n$ is converging to $u$.}

Let us prove that $u_+^\lambda:\R\times\Omega\to\R$ is measurable with respect to the product $\sigma$-algebra ${\mathcal B}\otimes\F$. This is equivalent to showing that $\Omega\ni \omega\mapsto u_+^\lambda(\,\cdot\,,\omega)\in\CC(\R)$ is a random variable from $(\Omega,\F)$ to the Polish space $\CC(\R)$ endowed with its Borel $\sigma$-algebra, see for instance  \cite[Proposition 2.1]{DS09}. 
Since the probability measure $\P$ is complete on $(\Omega,\F)$, it is enough to show that, for every fixed $\eps>0$, there exists a set $F\in \F$ with $\P(\Omega\setminus F)<\eps$ such that the restriction 
$u_+^\lambda$ to $F$ is a random variable from $F$ to $\CC(\R)$.  To this aim, we notice that the measure $\P$ is inner regular on $(\Omega,\F)$, see \cite[Theorem 1.3]{Bill99}, hence it is a Radon measure. By applying Lusin's Theorem \cite{LusinThm} to the random variables $a:\Omega\to\CC(\R)$ and $V:\Omega\to\CC(\R)$, we infer that there exists a closed set $F\subseteq \Omega$ with $\P(\Omega\setminus F)<\eps$ such that $a_{| F},V_{| F}:F\to\CC(\R)$ are continuous. 
We claim that $F\ni\omega\mapsto u_+^\lambda(\,\cdot\,,\omega)\in\CC(\R)$ is continuous. 
Indeed, let $(\omega_n)_{n\in\N}$ be a sequence converging to some $\omega_0$ in $F$. The functions $u_+^\lambda(\cdot,\omega_n)$ are equi-Lipschitz and locally equi-bounded in $\R$, hence they converge, up to subsequences, to a function $u$ satisfying $u(0)=0$ and \eqref{eq Lipschitz bound} for a.e. $x\in\R$. Since $a(\cdot,\omega_n)\to a(\cdot,\omega_0)$ and 
$V(\cdot,\omega_n)\to V(\cdot,\omega_0)$ in $\CC(\R)$, we derive by stability that $u$ solves \eqref{eq:cellODE} with $\omega:=\omega_0$ in the viscosity sense. By uniqueness, we infer that $u\equiv u_+^\lambda(\cdot,\omega_0)$, yielding in particular that the whole sequence $\big(u_+^\lambda(\cdot,\omega_n)\big)_n$ is converging to $u_+^\lambda(\cdot,\omega_0)$. This proves the asserted continuity property of the map $\omega\mapsto u_+^\lambda(\cdot,\omega)$ on $F$. 

Last, for every fixed $z\in\R$ and $\omega\in\Omega$, the map $v(\cdot):=u_+^\lambda(\cdot+z,\omega)-u_+^\lambda(z,\omega)$ is a solution of equation \eqref{eq:cellODE} with $\tau_z\omega$ in place of $\omega$ (by the stationary character of $a$ and $V$). Furthermore, it satisfies $v(0)=0$ and \eqref{eq Lipschitz bound} for a.e. $x\in\R$.  By uniqueness, we get $v=u_+^\lambda(\cdot,\tau_z\omega)$, i.e., 
$u_+^\lambda(\cdot+z,\omega)-u_+^\lambda(z,\omega)=u_+^\lambda(\cdot,\tau_z\omega)$. This clearly implies that $u_+^\lambda$ has stationary derivative. 
\end{proof}

To ease notation, in what follows we shall denote by $\partial_x u_+^\lambda(\cdot,\omega)$ and 
$\partial_x u_-^\lambda(\cdot,\omega)$ the derivatives with respect to $x$ of the functions $u_+^\lambda(\cdot,\omega)$ and 
$u_-^\lambda(\cdot,\omega)$ in $\R$.

\begin{theorem}\label{teo2 correctors}
Let $\lambda_2,\lambda_1\in\R$ with 
$\lambda_2>\lambda_1\geqslant\beta$, $R:=\max\{-G_-^{-1}(\lambda_2),G_+^{-1}(\lambda_2)\}$ and $C_R$ be a Lipschitz constant 
of $G$ in $[-R,R]$. Then there exists a set $\hat\Omega$ of probability 1 such that, for every $\omega\in\hat\Omega$, the following holds:\smallskip
\begin{itemize}
\item[(i)]\quad $\dfrac{\lambda_2-\lambda_1}{C_R}\leqslant \partial_x u_+^{\lambda_2}(\cdot,\omega)-\partial_x u_+^{\lambda_1}(\cdot,\omega)
\leqslant 
\dfrac{\lambda_2-\lambda_1}{\eta}
\qquad
\hbox{a.e. in $\R$}$;\medskip
\item[(ii)]\quad $\dfrac{\lambda_2-\lambda_1}{C_R}\leqslant \partial_x u_-^{\lambda_1}(\cdot,\omega)-\partial_x u_-^{\lambda_2}(\cdot,\omega)\leqslant 
\dfrac{\lambda_2-\lambda_1}{\eta}
\qquad
\hbox{a.e. in $\R$}$.
\end{itemize}
\end{theorem}

\begin{proof}
Let us prove item (i). According to the proof of Theorem \ref{teo1 correctors}, for every $\lambda\geqslant \beta$ and $\omega\in\Omega$, the function $u^\lambda_+(\cdot,\omega)$ can be obtained as the limit in $\CC(\R)$, for $n\to +\infty$, of the unique solution ${u^\lambda_n}_+(\cdot,\omega)$ of equation \eqref{eq:cellODE} with 
$a(\cdot,\omega)\vee 1/n$  in place of $a(\cdot,\omega)$ which satisfies ${u^\lambda_n}_+(0,\omega)=0$ and \eqref{eq Lipschitz bound}. 
Since the inequalities in (i) are preserved in the limit, it suffices to prove (i) by additionally assuming $a>0$ in $\R\times\Omega$.  In this case, the functions 
$u^{\lambda_1}_+(\cdot,\omega)$ and $u^{\lambda_2}_+(\cdot,\omega)$ are,  for every $\omega$, of class $C^2$ in $\R$ by Proposition \ref{prop regularity solutions}. Let us set $f_1:=\partial_x u^{\lambda_1}_+$ and $f_2:=\partial_x u^{\lambda_2}_+$. 
Then each function $f_i$ is a stationary and bounded and is a classical solution the following ODE, for every $\omega\in\Omega$:
\begin{equation}\label{eq ODE i}
a(x,\omega)f_i'(x,\omega)+G(f_i(x,\omega))+\beta V(x,\omega)=\lambda_i\qquad\hbox{in $\R$.}
\end{equation}
According to \cite[Lemma A.5]{DKY23}, we have 
\[ 
\P((f_1-f_2)(x,\omega) \geqslant (\lambda_2-\lambda_1)/{C_R}\ \forall x\in\R)=1
\quad
\text{or}
\quad
\P((f_2-f_1)(x,\omega) \geqslant (\lambda_2-\lambda_1)/{C_R}\ \forall x\in\R) = 1. 
\]
If the first alternative occurs, from \eqref{eq ODE i} and the fact that $f_1>f_2\geqslant 0$ we infer 
\[
(f_2-f_1)'(x,\omega)
\geqslant
(f_2-f_1)'(x,\omega)+G(f_2(x,\omega))-G(f_1(x,\omega))
=
\lambda_2-\lambda_1>0
\qquad
\hbox{in $\R$,}
\]
in contradiction with the fact that $f_2-f_1$ is bounded.  Hence the second alternative occurs and the first inequality in item (i) follows. Let us prove the second inequality in (i). According to \cite[Lemma A.2]{DKY23}, there exists a set $\hat\Omega$ of probability 1 such that, for every $\omega\in\hat\Omega$, the function $(f_2-f_1)(\cdot,\omega)$ has infinitely many local maxima and minima. Let us fix $\omega\in\hat\Omega$ and let $x_0\in\R$ be a local maximum of $(f_2-f_1)(\cdot,\omega)$. Then 
$(f_2-f_1)'(x_0,\omega)=0$, hence from \eqref{eq ODE i} and the fact that $f_2>f_1\geqslant 0$ we infer  
\[
\lambda_2-\lambda_1=G(f_2(x_0,\omega))-G(f_1(x_0,\omega))
\geqslant 
\eta \left(f_2(x_0,\omega)-f_1(x_0,\omega)\right).
\]
By the arbitrariness of the choice of the local maximum $x_0$ in $\R$ and of $\omega$ in $\hat \Omega$, we conclude that 
\[
\sup_\R \left(f_2(\cdot,\omega)-f_1(\cdot,\omega)\right)\leqslant \dfrac{(\lambda_2-\lambda_1)}{\eta}
\qquad
\hbox{for all $\omega\in\hat\Omega$.}
\]
The proof of item (ii) is analogous and is omitted.
\end{proof}

We conclude this section with a technical result that will play a crucial role in the proof of the homogenization result. The proof 
exploits the scaled hill condition (S), which is used here for the first time (apart from the result in \cite{DK22} recalled in Proposition \ref{prop lower bound}). 

\begin{lemma}\label{lem:cerca2}
Let us assume that conditions (A),(V) and (S) are in force, and $G\in\Hamzero$. Suppose that the set $\{x\in\R\,:\,a(x,\omega)=0\}$ has zero Lebesgue measure, almost surely. For every $\eps > 0$ and $y_0 > 0$, there exists a set $\widehat\Omega=\widehat\Omega(\eps,y_0)$ of probability 1 such that, for every $\omega\in\widehat\Omega$ and for $u_+\in\Sol_+(\beta,\omega)$, $u_-\in\Sol_-(\beta,\omega)$, the following holds:\medskip
\begin{itemize}
\item[] there exists $\ell_1<L_1<L_2<\ell_2$ such that \ $a(\cdot,\omega)>0$ in $(\ell_1,\ell_2)$, $u_\pm\in C^2\big((\ell_1,\ell_2)\big)$,
	\[ \int_{L_1}^{L_2}\frac{dx}{a(x,\omega)} \geqslant y_0
	\quad\text{and}\quad
	0\leqslant u_+'(x) - u_-'(x) \leqslant \eps\ \ \forall x\in[L_1,L_2].
	\]
\end{itemize}
\end{lemma}

\begin{proof}
The proof is essentially the same as the one of Lemma 4.3 in \cite{DKY23}. For consistency of notation, we set $p_1:=G_-^{-1}(\beta)$, $p_2:=0$ and $p_3:=G_+^{-1}(\beta)$. 

Fix $h\in\left(1 - \frac{\min\{G(-\eps),G(\eps)\} }{2\beta},1\right)$ so that
	\begin{equation}\label{eq:hczar}
		\beta(1- h) - G(\pm\eps) < -\beta(1-h).
	\end{equation}
Without loss of generality, assume that
	\begin{equation}\label{eq:nohayzar}
		y_0 > \frac{\max\{- p_1, p_3\} }{\beta(1-h)}.
	\end{equation}
By the scaled hill condition (S), we can find a set $\Omega(h,3y_0)$ of probability $1$ such that, for every $\omega\in\Omega(h,3y_0)$, 
(S)(a) and (S)(h) hold with $y:= 3y_0$, for a suitable choice of $\delta>0$ and of $\tilde\ell_1<\tilde\ell_2$. 
Up to discarding a subset of zero probability, we can furthermore assume that, for every $\omega\in\Omega(h,3y_0)$, the set $\{x\in\R\,:\,a(x,\omega)=0\}$ has zero Lebesgue measure. 

Let us fix $\omega\in\Omega(h,3y_0)$ and denote by $(\ell_1,\ell_2)$ a connected component of the open set  $(\tilde\ell_1,\tilde\ell_2)\cap\{x\in\R\,:\,a(x,\omega)>0\}$. Note that 
\begin{equation}\label{eq integral inequality}
\int_{\ell_1}^{\ell_2}\frac{dx}{a(x,\omega)}
\geqslant
 \int_{\tilde\ell_1}^{\tilde\ell_2}\frac{dx}{a(x,\omega)\vee\delta}
=3y_0. 
\end{equation}
Indeed, if $(\ell_1,\ell_2)\subsetneq(\tilde\ell_1,\tilde\ell_2)$, then either $a(\ell_1,\omega)=0$ or $a(\ell_2,\omega)=0$. By arguing as in the proof of Lemma \ref{lemma infinite integral}, we get that the left-hand side of \eqref{eq integral inequality} is equal to $+\infty$. Let us pick $L_1,L_2\in(\ell_1,\ell_2)$ such that 
	\begin{equation}\label{eq:zelizar}
		\int_{\ell_1}^{L_1}\frac{dx}{a(x,\omega)}\geqslant y_0,
		\qquad 
		\int_{L_1}^{L_2}\frac{dx}{a(x,\omega)}\geqslant y_0,
		\qquad
		 \int_{L_2}^{\ell_2}\frac{dx}{a(x,\omega)}\geqslant y_0.
	\end{equation}
Since $a(\cdot,\omega)>0$ in $(\ell_1,\ell_2)$, we know from Proposition \ref{prop regularity solutions} that the functions $u_\pm$ are of class $C^2$ in $(\ell_1,\ell_2)$. Let us set $f_1:=u'_-(\cdot,\omega)$ and 
$f_2:=u_+'(\cdot,\omega)$. 
For every $\omega \in \Omega(h,3y_0)$, if $f_2(x) \geqslant \eps$ for some $x\in(\ell_1,\ell_2)$, then
	\begin{equation}\label{eq:baltazar}
		f_2'(x) = \frac{\lambda - \beta V(x,\omega)-G(\eps)}{a(x,\omega)}\leqslant\frac{\lambda - \beta h - G(\eps)}{a(x,\omega)} < -\frac{\beta (1-h)}{a(x,\omega)}
	\end{equation}
by  \eqref{eq:hczar} and (S)(h).
It follows that $0 \leqslant f_2(x_2) \leqslant  \eps$ for some $x_2\in(\ell_1,L_1]$. Indeed, otherwise, for every 
$y\in (\ell_1,L_1)$ we have 
\begin{align*}
f_2(L_1) = f_2(y) + \int_{y}^{L_1} f_2'(x)\,dx & < p_3 - \beta(1-h)\int_{y}^{L_1} \frac{dx}{a(x,\omega)}.
\end{align*}
By sending $y\to \ell_1^+$ we get \ $f_2(L_1)\leqslant p_3 - \beta(1-h)y_0 < 0$, 
by \eqref{eq:nohayzar} and \eqref{eq:zelizar}, contradicting the fact that $f_2\geqslant 0$ in $(\ell_1,\ell_2)$. 
Via the same argument, we derive from \eqref{eq:zelizar} and \eqref{eq:baltazar} that $0 \leqslant f_2(x) \leqslant \eps$ for every $x\in[x_2,\ell_2) \supset [L_1,L_2]$.
	
	Similarly, for every $\omega \in \Omega(h,3y_0)$, there exists an $x_1\in[L_2,\ell_2)$ such that $- \eps \leqslant f_1(x_1) \leqslant 0$ for every $x\in(\ell_1,x_1] \supset [L_1,L_2]$. This can be shown by repeating the argument in the paragraph above with
	\[ \check G(p) = G(-p),\quad \check a(x,\omega) = a(-x,\omega),\quad \check V(x,\omega) = V(-x,\omega)\quad\text{and}\quad f_2(x) = -f_1(-x). 
\]
The assertion follows by the arbitrariness of the choice of $\eps>0$. 
\end{proof}

\section{The homogenization result}\label{sec:homogenization}

In view of Theorem \ref{teo1 correctors}, for every $\lambda\geqslant \beta$ we can define the following quantities:
\[
\theta_+(\lambda):=\E[\partial_x u_+^\lambda(0,\omega)],
\qquad
\theta_-(\lambda):=\E[\partial_x u_-^\lambda(0,\omega)].
\]
We remark that such quantities are well defined and do not depend on the fact that we are evaluating the derivative of $u^\lambda_\pm(\cdot,\omega)$ at $x=0$. Indeed, for every fixed $y\in\R$, there exists a set $\Omega_y$ of probability 1 such that the functions $u_+^\lambda(\cdot,\omega)$ and $u_-^\lambda(\cdot,\omega)$ are differentiable at $y$ for every $\omega\in\Omega_y$. Furthermore, by stationarity, we have 
\[
\E[\partial_x u_+^\lambda(y,\omega)]=\E[\partial_x u_+^\lambda(0,\omega)]
\qquad
\hbox{and}
\qquad
\E[\partial_x u_-^\lambda(y,\omega)]=\E[\partial_x u_-^\lambda(0,\omega)],
\]
see \cite[Proposition 3.6]{DS09} for a proof.\smallskip 
\begin{prop}\label{prop properties theta}
The maps $\lambda\mapsto \theta_\pm(\lambda)$ from $[\beta,+\infty)$ to $\R$ are Lipschitz continuous, coercive and strictly monotone. 
More precisely, the following holds:\medskip
\begin{itemize}
\item[(i)] $\theta_- (\lambda) \in  [G_-^{-1}(\lambda), G_-^{-1}(\lambda-\beta)]$, \quad 
$\theta_+ (\lambda) \in[G_+^{-1}(\lambda-\beta), G_+^{-1}(\lambda)]$ \quad for every $\lambda\geqslant \beta$;\medskip
\item[(ii)]  for every $\lambda_2>\lambda_1\geqslant\beta$ we have 
\begin{align*}
\dfrac{\lambda_2-\lambda_1}{C_R}
\leqslant
\theta_-(\lambda_1)-\theta_-(\lambda_2)
\leqslant \dfrac{\lambda_2-\lambda_1}{\eta},\qquad
\dfrac{\lambda_2-\lambda_1}{C_R}
\leqslant
\theta_+(\lambda_2)-\theta_+(\lambda_1)
\leqslant \dfrac{\lambda_2-\lambda_1}{\eta},
\end{align*}
where $R:=\max\{G_-^{-1}(\lambda_2),G_+^{-1}(\lambda_2)\}$ and $C_R$ be a Lipschitz constant 
of $G$ in $[-R,R]$.\smallskip
\end{itemize}
Furthermore, 
\[
\HV^L(G)(\theta_{\pm}(\lambda))
=
\HV^U(G)(\theta_{\pm}(\lambda))
=
\lambda
\qquad
\hbox{for every $\lambda\geqslant \beta$.}
\]
\end{prop}
\begin{proof}
Assertions (i) and (ii) are a direct consequence of the definition of $\theta_\pm(\cdot)$ and of Theorem \ref{teo2 correctors}. 
The last assertion follows from Proposition \ref{prop consequence existence corrector}. 
\end{proof}

We have thus shown that 
\[
\HV^L(G)(\theta)=\HV^U(G)(\theta)
\qquad
\hbox{for all $\theta \in  (-\infty,\theta_-(\beta)]\cup [\theta_+(\beta),+\infty)$.}
\]
It is left to show that the above equality holds for all $\theta \in \big(\theta_-(\beta), \theta_+(\beta)\big)$ when this set is 
nonempty. Since correctors corresponding to such a $\theta$ need not exist, we have to proceed differently, by exploiting an idea that was introduced in \cite[Lemma 4.7]{Y21b}. The argument has been subsequently refined and substantially simplified 
in \cite[Section 3]{DKY23}. Both these works deal with the nondegenerate case $a(\cdot)>0$. We propose here a variant of this argument that takes into account the possibility that $a=0$ at some points.  The proof does not use directly the scaled hill condition (S), but it crucially depends on it via Proposition \ref{prop lower bound} and Lemma \ref{lem:cerca2}. 

\begin{prop}\label{prop flat part} 
We have \ $\HV^L(G)(\theta)=\HV^U(G)(\theta)=\beta$\ for every $\theta\in \big(\theta_-(\beta),\theta_+(\beta)\big)$.
\end{prop}

\begin{proof}
Let us assume that $\theta_-(\beta)<\theta_+(\beta)$, being the assertion otherwise empty. 
In view of Propositions \ref{prop lower bound}  and \ref{prop properties theta} above, it suffices to show that 
$\HV^U(G)(\theta)\leqslant \beta$ for every $\E[\partial_x u_-^\beta(0,\omega)]<\theta<\E[\partial_x u_+^\lambda(0,\omega)]$. 

Let us fix such a $\theta$, and denote by $\ul f,\ol f$ two functions in $L^\infty(\R\times\Omega)$ satisfying, for every 
$\omega\in\Omega$,
\[
\ul f(\cdot,\omega)=\partial_x u^\beta_-(\cdot,\omega),\qquad 
\ol f(\cdot,\omega)=\partial_x u^\beta_+(\cdot,\omega)
\qquad\hbox{a.e. in $\R$.}
\]
Pick $R>\max\{\|\ul f\|_{L^\infty(\R\times\Omega)} ,\|\ol f\|_{L^\infty(\R\times\Omega)}\}$ and denote by $C_R$ the Lipschitz constant of $G$ on $[-R,R]$. Let us fix $\eps>0$. 
The first part of the proof consists in showing the existence of a set $\widehat\Omega$ of probability 1 such that, for every $\omega\in\widehat\Omega$, there exists a function $w_\eps\in\Lip(\R)$ satisfying the following inequality in the viscosity sense
\begin{equation}\label{eq supersolution}
a(x,\omega)u''+G(u')+\beta V(x,\omega)\leqslant \beta+2(C_R+2)\eps\qquad\hbox{in $\R$}, 
\end{equation}
and such that $w_\eps'=\ul f(\cdot,\omega)$ in $(-\infty,-L)$,  $w_\eps'=\ol f(\cdot,\omega)$ in $(L,+\infty)$ for $L>0$ big enough. 

Let us start with the case when the set $\{x\in\R\,:\,a(x,\omega)=0\}$ has positive Lebesgue measure for every $\omega$ in a set $\widehat\Omega$ of probability 1. 
Then, for every fixed $\omega\in\widehat\Omega$, there exists a point $x_0\in\R$ such that $a(x_0,\omega)=0$ and $u^\beta_-(\cdot,\omega)$, $u^\beta_+(\cdot,\omega)$ are both differentiable at $x_0$. 
Note that $\partial_x u^\beta_-(x_0,\omega) \leqslant \partial_x u^\beta_+(x_0,\omega)$ by definition of $u^\beta_-$ and $ u^\beta_+$. 
In view of Proposition \ref{prop pointwise solution} we have 
\[
G\big(\partial_x u^\beta_\pm(x_0,\omega)\big)+\beta V(x_0,\omega)=\beta.
\]
We define $w_\eps$ by setting 
\[
w_\eps(x)=\int_{x_0}^x \Big( \ul f(z,\omega)\1_{(-\infty,x_0]}(z)+\ol f(z,\omega)\1_{(x_0,+\infty)}(z) \Big)\,dz,
\qquad
x\in\R.
\footnote{We denote by $\1_E$ the characteristic function of the set $E$, i.e., the function which is identically equal to 1 on $E$ and to 0 in its complement.}
\]
Then $w_\eps=u_-^\beta(\cdot,\omega)-u_-^\beta(x_0,\omega)$ in $(-\infty,x_0)$ and 
$w_\eps=u_+^\beta(\cdot,\omega)-u_+^\beta(x_0,\omega)$  in $(x_0,+\infty)$, hence $w_\eps$ solves \eqref{eq:cellODE} with $\lambda:=\beta$ in $\R\setminus\{x_0\}$. On the other hand, the subdifferential of $w_\eps$ the point $x_0$ is  
$D^-w_\eps(x_0)=[\partial_x u^\beta_-(x_0,\omega),\partial_x u^\beta_+(x_0,\omega) ]$. Since any $C^2$--subtangent $\varphi$ to $w_\eps$ at $x_0$ satisfies $\varphi'(x_0)\in D^-w_\eps(x_0)$, by the quasiconvex character of $G$ we infer  
\begin{eqnarray*}
a(x_0,\omega)\varphi''(x_0)+G\big(\varphi'(x_0)\big)&+&\beta V(x_0,\omega)
=G\big(\varphi'(x_0)\big)+\beta V(x_0,\omega)\\
&\leqslant& 
\max\left\{G\big(\partial_x u^\beta_-(x_0,\omega)\big), G\big(\partial_x u^\beta_+(x_0,\omega)\big)\right\}+\beta V(x_0,\omega)
=
\beta.
\end{eqnarray*}
We conclude that $w_\eps$ satisfies \eqref{eq supersolution}. \smallskip

Let us then assume that the set $\{x\in\R\,:\,a(x,\omega)=0\}$ is almost surely of zero Lebesgue measure. 
Let us choose $y_0\geqslant 2$, and denote by $\widehat\Omega=\widehat\Omega(\eps,y_0)$ the set of probability 1 chosen according to Lemma \ref{lem:cerca2}. Up to discard a subset of zero probability, we can furthermore assume that 
$\{x\in\R\,:\,a(x,\omega)=0\}$ has zero Lebesgue measure for every $\omega\in\widehat\Omega$. 

Let us fix $\omega\in\widehat\Omega$ and let us choose $\ell_1<L_1<L_2<\ell_2$ as in the statement of Lemma  \ref{lem:cerca2}. Since the functions $u_\pm^\beta$ are of class $C^2$ in $(\ell_1,\ell_2)$, we infer that the functions $\ol f$, $\ul f$ are classical solutions of the following ODE: 
\begin{equation}\label{eq: ODE beta}
a(x,\omega)f'(x)+G(f(x))+\beta V(x,\omega) = \beta\qquad\hbox{in  $(\ell_1,\ell_2)$.}
\end{equation}
We now proceed by arguing as in \cite[Section 3]{DKY23}.  
Pick $r>0$ small enough so that $[L_1 -2 r,L_2 + 2r]\subset (\ell_1,\ell_2)$. 
We shall construct a function $f_\eps(\,\cdot\,,\omega)\in\CC^1(\R)$ by interpolating $\ul f$ and $\ol f$ on $[L_1 - r,L_2 + r]$ as follows:
	\begin{equation}\label{eq:zand}
		f_\eps(x,\omega)=(1-\xi(x,\omega))\ul{f}(x,\omega)+\xi(x,\omega)\ol{f}(x,\omega),
	\end{equation}
	where $\xi(\,\cdot\,,\omega)\in\CC^1(\R)$ is non-decreasing,
	\begin{equation}\label{eq:specs}
		\xi(x,\omega)\equiv 0\ \text{on}\ (-\infty, L_1-r] \quad \text{and} \quad \xi(x,\omega)\equiv 1\ \text{on}\ [L_2+r,+\infty).
	\end{equation}
Note that $\xi'(\,\cdot\,,\omega)$ is supported on $[L_1-r,L_2+r]$. Our goal is to find such a $\xi(\,\cdot\,,\omega)$ such that
	\begin{equation}\label{eq:ecza}
		a(x,\omega)f_\eps'(x,\omega) + G(f_\eps(x,\omega)) + \beta V(x,\omega) \leqslant \lambda +2(C_R+1)\eps
		\qquad
		\hbox{for all $x\in\R$.}
	\end{equation}
	For notational simplicity, we momentarily suppress $(x,\omega)$ from the notation and observe that
	\begin{align}\label{eq:concur}
		af_\eps' + G(f_\eps) + \beta V &= (1-\xi)(a\ul{f}' + G(\ul{f}) + \beta V) + \xi(a\ol{f}' + G(\ol{f}) + \beta V)\nonumber\\
		&\quad + G((1-\xi)\ul{f} + \xi\ol{f}) -(1- \xi) G(\ul{f}) - \xi G(\ol{f}) +a\xi'(\ol{f} - \ul{f})\nonumber\\
		&= \beta + G(\xi\ul{f}+(1-\xi)\ol{f}) - (1-\xi) G(\ul{f}) - \xi G(\ol{f}) +a\xi'(\ol{f}-\ul{f}).
	\end{align}
From \eqref{eq:specs} and \eqref{eq:concur} we see that \eqref{eq:ecza} holds if 
	\begin{equation}\label{eq:ifal}
		(\ol{f} - \ul{f})(x,\omega) < 4\eps \quad \forall x\in[L_1-r,L_2+r]
	\end{equation}
	and
	\begin{equation}\label{eq:husnu}
		a(x,\omega)\xi'(x,\omega) \leqslant 1, \quad \forall x\in[L_1-r,L_2+r].
	\end{equation}
Suppose tentatively that we take $\xi(x,\omega)=\int_{-\infty}^x \zeta(s,\omega)\,ds$, where
    \begin{equation}\label{eq:zeta}
    	\zeta(x,\omega)=(y_0a(x,\omega))^{-1}\mathbbm{1}_{(L_1,L_2)}(x).
    \end{equation}
Note that \eqref{eq:specs} and \eqref{eq:husnu} hold with $r = 0$. However, $\xi(\,\cdot\,,\omega)$ is not in $\CC^1(\R)$. We will solve the issue via a standard mollification argument. 
   
Since $a(\,\cdot\,,\omega) > 0$ in $(\ell_1,\ell_2)$, the difference $(\ol{f}-\ul{f})(\,\cdot\,,\omega)$ is locally Lipschitz due to \eqref{eq: ODE beta} and the fact that $\ul f,\ol f$ are bounded. Hence, up to choosing a smaller $r>0$ if necessary, we infer that \eqref{eq:ifal} holds, in view of Lemma \ref{lem:cerca2} . Take a standard sequence of even convolution kernels $\rho_n$ supported on $[-1/n,1/n]$. For all sufficiently large $n,\ n>1/r$, define $\zeta_n(x,\omega)=(\zeta(\,\cdot\,,\omega)*\rho_n)(x)$ with $\zeta(\,\cdot\,,\omega)$ as in \eqref{eq:zeta}, and let $\xi_n(x,\omega) = \int_{-\infty}^x\zeta_n(s,\omega)\,ds$. Note that $\xi_n(\,\cdot\,,\omega)$ satisfies \eqref{eq:specs}.
	Moreover, letting $\ul{a}(\omega) = \inf_{x\in[L_1-2r,L_2+2r]}a(x,\omega)>0$ and recalling that $y_0\geqslant 2$, the following inequalities hold:
	\begin{align*}
		a(x,\omega)\xi'_n(x,\omega)&\leqslant \frac{1}{y_0}\int_{-1/n}^{1/n}\frac{a(x,\omega)\rho_n(s)}{a(x+s,\omega)}\,ds\\
		&\leqslant \frac12+\frac12\int_{-1/n}^{1/n}\frac{(a(x,\omega)-a(x+s,\omega))\rho_n(s)}{a(x+s,\omega)}\,ds\leqslant \frac12+\frac{\kappa}{\ul{a}(\omega) n}
	\end{align*}
	for all $x\in [L_1 - r,L_2 + r]$, where $\kappa$ is the Lipschitz constant of $\sqrt{a(\,\cdot\,,\omega)}$ from (A).	Choosing $n > \max\left\{\frac1r,\frac{2\kappa}{\ul{a}(\omega)}\right\}$, we deduce that $\xi_n(\,\cdot\,,\omega)$ satisfies \eqref{eq:husnu}, hence \eqref{eq:ecza} follows by choosing $\xi(\,\cdot\,,\omega):= \xi_n(\,\cdot\,,\omega)$. It is easily seen that the function \ $w_\eps(x):=\int_0^{x} f_\eps(z)\,dz$ solves \eqref{eq supersolution}.\smallskip
	
We are ready to prove the assertion, following the argument provided in \cite[Lemma 5.6]{DK22}. There is a set $\wh\Omega_2 \subset \wh\Omega_1$ of probability 1 on which the limits in \eqref{eq:infsup} hold. For every $\omega\in\wh\Omega_2$, let 
\[
\tilde w_\eps(t,x,\omega)=(\beta+K\eps)t+w_\eps(x)+M_\eps(\omega),
\]
where $w_\eps(x)$ is the solution of \eqref{eq supersolution} defined in the first part of the proof,  $K:=2(C_R+2)$ and $M_\eps(\omega)$ will be chosen later to ensure that $\tilde w_\eps(0,x,\omega)\geqslant \theta x$ for all $x\in\R$. Note that $\tilde w_\eps(\,\cdot\,,\,\cdot\,,\omega)$ is a supersolution of \eqref{eq:generalHJ}. Indeed, 
	\begin{align*}
		\partial_t\tilde w_\eps
=
		(\beta+K\eps)
\overset{\eqref{eq:ecza}}{\ge}		
a(x,\omega)\partial_{xx}^2\tilde w_\eps + G(\partial_x \tilde w_\eps)+\beta V(x,\omega)	
\end{align*}
	By Birkhoff Ergodic Theorem and the choice of $\theta$, there exists an $\wh\Omega_3 \subset \wh\Omega_2$ of probability $1$ such that, for every $\omega\in \wh\Omega_3$,
	\[\limsup_{x\to+\infty}\frac{\tilde w_\eps(0,x,\omega)}{x}=\lim_{x\to+\infty}\frac{1}{x}\int_{L_2+r}^x\ol{f}(s,\omega)\,ds=\E[\ol f(0,\omega)] > \theta\]
	and
	\[\liminf_{x\to-\infty}\frac{\tilde w_\eps(0,x,\omega)}{x}=\lim_{x\to-\infty}\frac{1}{|x|}\int_{x}^{L_1 - r}\ul{f}(s,\omega)\,ds=\E[\ul f(0,\omega)] < \theta.\]
	Therefore, for every $\omega\in\wh\Omega_3$, we can pick $M_\eps(\omega)$ large enough so that $\tilde w_\eps(0,x,\omega)\geqslant \theta x$ for all $x\in\R$.  By the comparison principle, $\tilde w_\eps(t,x,\omega)\geqslant u_\theta(t,x,\omega)$ on $\ccyl\times\wh\Omega_3$ and, hence,
	\[\HV^U(G)(\theta)=\limsup_{t\to+\infty}\frac{u_\theta(t,0,\omega)}{t}\leqslant\limsup_{t\to+\infty}\frac{\tilde w_\eps(t,0,\omega)}{t}=\beta+K\eps\] 
	for every $\omega\in \wh\Omega_3$. Since $\eps > 0$ was arbitrarily chosen, we get the assertion.  
\end{proof}

We are now in position to prove Theorem \ref{thm:genhom}.

\begin{proof}[Proof of Theorem \ref{thm:genhom}]
In view of Theorem \ref{teo reduction}, we can assume that $G\in\Hamzero$. 
According to Proposition \ref{prop properties theta}, the effective Hamiltonian $\HV(G)$ associated to equation 
\eqref{eq:generalHJ} can be defined on 
$(-\infty,\theta_-(\beta)]\cup [\theta_+(\beta),+\infty)$ by taking the inverse of the maps $\lambda\mapsto \theta_\pm(\lambda)$. 
More precisely, for every $\theta\in (-\infty,\theta_-(\beta)]\cup [\theta_+(\beta),+\infty)$, there exists a
unique $\lambda\geqslant \beta$ such that either $\theta=\theta_-(\lambda)$ or $\theta_+(\lambda)$. 
For such a $\theta$, we set $\HV(G)(\theta)=\lambda$. 
We extend $\HV(G)$ to the whole $\R$ by setting $\ \HV^L(G)(\theta)=\HV^U(G)(\theta)=\beta$ for every $\theta\in \big(\theta_-(\beta),\theta_+(\beta)\big)$, according to Proposition \ref{prop flat part}.  

The fact that $\HV(G)$ is continuous is a direct consequence of \cite[Proposition 3.4]{DK17}, where condition (L$'$) is satisfied 
due to conditions (A),(V) and the fact that $G\in\Ham$, see for instance Theorem 2.8 and Remark 2.9 in \cite{DK17}. 
From the two inequalities appearing in the statement of Proposition \ref{prop properties theta}, we derive that $\HV(G)$ is locally Lipschitz on $\R$, strictly decreasing on $(-\infty,\theta_-(\beta)]$ and strictly increasing on $[\theta_+(\beta),+\infty)$. Hence 
$\HV(G)$ is quasiconvex.  
\end{proof}

\bibliography{viscousHJ}
\bibliographystyle{siam}
\end{document}